\documentclass{amsart}

\newtheorem{theorem}{Theorem}

\newtheorem{proposition}{Proposition}
\newtheorem{lemma}{Lemma}
\newtheorem{corollary}{Corollary}
\newtheorem*{assumption}{Assumption}

\theoremstyle{definition}

\theoremstyle{remark}
\newtheorem{remark}{Remark}

\numberwithin{equation}{section}

\newcommand{\RR}{\mathbb{R}}

\newcommand{\CC}{\mathbb{C}}

\newcommand{\C}{\mathbb{C}}

\newcommand{\ba}{\begin{array}}
\newcommand{\be}{\begin{equation}}
\newcommand{\ea}{\end{array}}
\newcommand{\ee}[1]{\label{#1}\end{equation}}

\newcommand{\codim}{{\rm {codim}}\,}

\newcommand{\Cal}[1]{{\mathcal {#1}}}
\newcommand{\ve}{{\mathcal {V}}}

\def\Id {{\rm \mbox{Id}}}
\def\codim {{\rm \mbox{codim\,}}}

\def\Ka {K\"ahler }

\begin{document}

\author{Sergio Console}
\address{Dipartimento di Matematica
Universit\`a di Torino,
via Carlo Alberto 10,
10123 Torino, Italy}
\email{sergio.console@unito.it}
\author{Antonio J. Di Scala}
\address{Dipartimento di Matematica, Politecnico di Torino
Corso Duca degli Abruzzi 24, 10129  Torino, Italy}
\email{antonio.discala@polito.it}
\thanks{First and second author were partially
supported by GNSAGA of INdAM, MIUR of Italy}
\author{Carlos Olmos}
\address{FaMAF, Universidad Nacional de C\'ordoba,
Ciudad Universitaria,
5000 C\'ordoba, Argentina}
\email{olmos@mate.uncor.edu}
\thanks{The third author was supported by Universidad Nacional de C\'ordoba and CONICET,
partially supported by   Antorchas, ANCyT, Secyt-UNC and CIEM}

\subjclass[2000]{Primary 53C30; Secondary 53C21}


\keywords{Submanifolds, holonomy, normal connection, orbit of isotropy representation, Hermitian symmetric space}

\title[A Berger type normal holonomy theorem]{A Berger type normal holonomy theorem for complex submanifolds}


\begin{abstract}
We prove a Berger type theorem for the normal holonomy $\Phi^\perp$ (i.e., the holonomy group of the normal connection) of  a full complete complex submanifold $M$  of the complex projective space $\C P^n$. Namely, if $\Phi^\perp$ does not act transitively, then $M$ is the complex orbit, in the complex projective space, of the isotropy
representation of an irreducible Hermitian symmetric space of rank greater or equal to $3$.
Moreover, we show that for complete irreducible complex submanifolds of $\CC^n$ the normal holonomy is generic, i.e.,  it  acts transitively on the unit sphere of the normal space. \newline
The methods in the proofs rely heavily on the singular data of appropriate
holonomy tubes (after lifting the submanifold to the complex Euclidean
space, in the $\C P^n$ case) and basic facts of complex submanifolds.
\end{abstract}

\maketitle

Berger's Holonomy Theorem \cite{B} is probably the most important general (local) result of Riemannian geometry: the restricted holonomy group of an irreducible Riemannian manifold acts transitively on the unit sphere of the tangent space except in the case that the manifold is a symmetric space of rank bigger or equal to two.

\smallskip

In submanifold geometry a prominent r\^ole is played by the holonomy group of the natural connection of the normal bundle, the so-called \textit{normal holonomy group}.

\smallskip

For submanifolds of $\RR^n$ or more generally of spaces of constant curvature,
a fundamental result is the \textit{Normal Holonomy Theorem} \cite{O1}. It asserts roughly that the non-trivial component of the action of the normal holonomy group on any normal space is the isotropy representation of a Riemannian symmetric space (called $s$-representation for short). The Normal Holonomy Theorem is a very important tool for the study of submanifold geometry, especially in the context of  submanifolds with ``simple extrinsic geometric invariants'', like isoparametric and homogeneous submanifolds (see \cite{BCO} for an introduction to this subject). In particular, in this extrinsic setting, some distinguished class of homogeneous submanifolds, the orbits of $s$-representations, play a similar r\^ole as symmetric spaces in intrinsic Riemannian geometry. Typically, requiring that a submanifold has ``simple extrinsic geometric invariants'' (e.g. ``enough'' parallel normal fields with respect to which the shape operator has constant eigenvalues) implies that the submanifold belongs to this class. Therefore, these methods based on the study of normal holonomy allowed to prove many results for  submanifolds with ``simple extrinsic geometric invariants'' \cite{BCO, CDO, CO, O2, O3, O4, PT, T, T2, Th}. But, actually, they turned out to be useful in (intrinsic) Riemannian geometry, as basic tools for a geometric proof of Berger's Theorem \cite{O5}.

\smallskip

In \cite{AD} the normal holonomy group   of  complex submanifolds of a complex space form was studied. It was proven that if the normal holonomy group acts irreducibly on the normal space then it is linearly isomorphic to the holonomy group of an irreducible Hermitian symmetric space. Moreover the normal holonomy group acts irreducibly if the submanifold is full (that is, it is not contained in a totally geodesic proper complex submanifold) and the second fundamental form at some point has no nullity.

\medskip

In the present paper, we prove a Berger type theorem for complex submanifolds of the complex projective space $\C P^n$.

\begin{theorem}\label{berger-cp}
Let $M$ be a  full and complete complex
projective submanifold of $\C P^n$. Then the following are equivalent:
\begin{enumerate}
\item The normal holonomy is not transitive on the unit sphere of the normal space (i.e., different from $U(k)$, $k=\codim(M)$, since it is
an s-representation).
\item $M$ is the complex orbit, in the complex projective space, of the isotropy
representation of an irreducible Hermitian symmetric space of rank greater or equal to $3$.
\end{enumerate}
\end{theorem}

Notice that $(2) \Rightarrow (1)$ was proved in \cite{CD}. It is well-known that
\textit{the complex orbit $M$, in the complex projective space, of the isotropy
representation of an irreducible Hermitian symmetric space is extrinsic symmetric} or equivalently its second fundamental form is parallel \cite{CD}. Thus, \textit{a full and complete complex submanifold $M \subset \C P^n$ whose normal
holonomy group is not transitive on the unit sphere of the normal space has parallel second fundamental form}.

\smallskip

For complete complex submanifolds of $\CC^n$ we prove that the normal holonomy is generic

\begin{theorem}\label{berger-c}
The normal holonomy group of  a complete  irreducible and full immersed complex submanifold of $\CC^n$  acts transitively on the unit sphere of the normal space. Indeed, $\Phi ^\perp=U(k)$, where $k$ is the codimension of the submanifold.
\end{theorem}

The proofs of the above results will be given in Sections~\ref{3} and \ref{4} respectively.

\smallskip

The completeness assumption cannot be dropped either in  Theorem \ref{berger-cp} or in  Theorem \ref{berger-c} (see Section~\ref{further}).

\medskip

The main tool is the study of the full holonomy tube (i.e., a holonomy tube with flat normal bundle, see \S~\ref{1.5}) of an Euclidean submanifold $N$ whose  normal holonomy group acts irreducibly and non-transitively on the unit sphere of the normal space. Let $M=N_\zeta$ be the full holonomy tube on $N$. We define a canonical foliation of $M$ whose leaves are holonomy tubes of some focal manifold as well.  It comes out there is a strong similarity with polar actions. Indeed, the orthogonal distribution to the holonomy tubes is integrable and its leaves behave like sections in a polar representation (Proposition~\ref{slice}). To show that the leaves of the canonical foliation are orbits of an isotropy representation ($s$-representation) we assume the horizontal distribution of a full holonomy tube is covered by kernels of shape operators. This implies that $M$ and $N$ are foliated by holonomy tubes around isotropy orbits (Theorem~\ref{main-lemma}).

\medskip

In order to apply this setting to a complete  irreducible and full immersed complex submanifold $M$ of the complex Euclidean space $\C^n$, we notice that if the  normal holonomy does not act transitively on the unit sphere of the normal space, then there are abundantly many kernels of shape operators in order to cover the horizontal distribution of a full holonomy tube.
Hence $M$ is foliated by holonomy tubes around orbits of the isotropy representation of a Hermitian symmetric space (Theorem~\ref{main-cx}).  Theorem~\ref{berger-c} is then a consequence of this result and the fact that  the normal holonomy group of a complex  irreducible full submanifold of $\CC^n$ acts irreducibly on the normal space \cite {DS}.

\medskip

Coming to complex submanifolds of $M \subset \C P^n$, we will lift $M$ to a submanifold $\tilde M$ of $\C^{n+1}\backslash\{ 0\}$. A key point in order to prove Theorem~\ref{berger-cp}  is then showing that the normal holonomy of  $\tilde M$ is not transitive on the unit sphere of the normal space if this is the case for $M$. This will be done in Section~\ref{4}.

\section{Preliminaries}\label{1}

We begin recalling some basic facts, which are now well-known, and have been extensively used by the authors in their work on submanifold geometry.
For most of the proofs we refer to \cite{BCO}.

\subsection{General notation and basic facts.}\label{1.1}

Let $M\subset \RR^n$ be a Euclidean submanifold, with induced metric $\langle \, , \, \rangle$ and  Levi-Civita connection $\nabla $.
We will use the notation $\nabla ^E$ for the  Levi-Civita connection in $\RR^n$.

We will always denote by  $\nu M := TM^\perp$  the normal bundle of $M$ endowed with
 the normal connection $\nabla ^\perp$. The maximal parallel and flat subbundle of $\nu M$
 will be written as $\nu _0 M$. Since we
 are working locally, all manifolds will be assumed to be simply connected. Hence $\nu _0 M$ is globally flat, that is, $\nu_0 M$ is spanned by the parallel normal fields to $M$. The normal curvature tensor will be denoted
 by $R^\perp$. We have that $R^\perp _{X,Y}\xi =0$, for $X$, $Y$ tangent fields and $\xi$ normal field lying in
 $\nu_0 M$.

\smallskip

The second fundamental form (with respect to the ambient
 Euclidean space) will be denoted by  $\alpha $  and the associated shape operator by $A$. These two  tensors are related by the well known formula, for
 any
 $X,\, Y$ tangent fields  and $\xi$ normal field,
 $\langle \alpha (X,Y), \xi \rangle = \langle A_\xi X, Y\rangle$,
 which is symmetric in $X, \, Y$.

 When there are several submanifolds involved, and it is
 not clear form the context, we add an upper script $M$, e.g. $\alpha ^M$ or $A^M$.

The connection $\nabla \oplus \nabla ^\perp$ of $TM\oplus\nu M$ will be denoted by $\bar
 {\nabla}$.

\medskip

 We recall the well-known formulae relating the
 basic objects in submanifolds geometry
   $$\langle R_{X,Y}Z, W\rangle = \langle \alpha (X,W), \alpha (Y,Z) \rangle - \langle \alpha (X,Z), \alpha (Y,W)
 \rangle\, , \leqno{\textit{(Gauss)}}$$
  $$\langle (\bar \nabla_XA)_\xi Y , Z\rangle \text {, or equiv., }  (\bar\nabla _X\alpha)(Y,Z)
   \text {  are symmetric in } X,\, Y,\,  Z \, , \leqno{\textit{(Codazzi)}}$$
$$\langle R^\perp_{X,Y}\xi ,\eta \rangle = \langle [A_\xi , A_\eta]X, Y\rangle\, .
\leqno{\textit{(Ricci)}}$$ \medskip
 Let $N$ be another
submanifold such that $M\subset N\subset \RR^n$. We say that $M$ is \textit{invariant under the shape operator}
$A^N$, briefly $M$ is $A^N$-invariant, if $A^N_\eta (T_xM)\subset T_xM$ for all $x\in M$, $\eta \in  \nu _xM$.
Equivalently, $\alpha ^N (T_xM, \nu_xM\cap T_xN) =0$. Observe that in this case $\nu N_{|M}$ is a parallel
subbundle of $\nu M$.

A distribution $\Cal D$ of $N$ is called $A^N$-invariant if $A^N_\eta (\Cal D_x)\subset \Cal D_x$, for all $x\in
N$, $\eta \in  \nu _x N$.

\medskip

The linear subspace of $T_pM$
\[{\mathcal N}_p=\bigcap_{\xi \in \nu_pM} \ker
A_\xi = \{ X_p \in T_pM : \alpha^M(\cdot, X) = 0 \}
\]
 is called the \textit{nullity space of $M$ at $p$}.
The collection of all these spaces is called the \textit{nullity
distribution}
of $M$. Note that this is
actually a distribution only on any connected component
of a suitable dense and open subset of $M$.

\medskip

 The \textit{normal exponential} of the Euclidean submanifold
$M$, $\exp _\nu : \nu M \to \RR^n$,  is defined by $\exp _\nu (\xi_p) = p  + \xi_p$. We set
$$\nu ^rM = \{ \xi \in \nu M:\ ||\xi|| < r\}\text { ,} \ \ \ \ \ \ \ \ \ \ \ \ \
 S_r \nu M = \{ \xi \in \nu M:\ ||\xi|| = r \}\, .$$
If $r$ is small, by making $M$ possibly smaller around a  point $q$, $\exp _\nu : \nu ^r M \to \RR^n$ is a diffeomorphism onto its image. In this case the so-called
spherical $\epsilon$-tube around $M$, denoted by $\exp _\nu (S_\varepsilon \nu  M)$, is a submanifold of $ \RR^n$, for all $\varepsilon < r$.
\medskip

The submanifold $M\subset \RR^n$ is said to be \textit{full} if it is not contained in any proper affine
subspace of the ambient space. The submanifold $M$ is said  to be locally  \textit{reducible} if one can write
locally $M = M_1\times M_2$ where $M_1 \subset \RR^{k}$, $M_2 \subset \RR^{n-k}$ and $\RR^n$ decomposes
orthogonally as $\RR^{k}\times \RR^{n-k}$. We say that $M$ is locally \textit{irreducible} if it is not
locally reducible. There are two very useful tools for deciding whether a submanifold $M$ of Euclidean space is
not  full or reducible.

\begin{enumerate}
\item  $M$ is not full if and only if there exists parallel normal field $\xi\neq 0 $ such that $A_\xi\equiv 0$.
\item  \textit{Moore's lemma}. $M$  is locally reducible if and only if  there exists a non
 trivial $A$-invariant parallel distribution of $M$.
\end{enumerate}

Let $X^n = G/K$ be a simply connected complete symmetric space without Euclidean de Rham factor, where $G$ is
the connected component of the full group of isometries of $X$. The isotropy representation of $K$ in the
Euclidean space $T_{[e]}X \simeq \RR^n$ is called an \textit{$s$-representation}. Any principal orbit $ M = K.v$ is an
\textit{isoparametric submanifold} of $\RR^n$. Namely, $\nu M$ is globally flat and $A_\xi$ has constant eigenvalues
for any parallel normal field $\xi$ to $M$. More in general, if $M$ is not necessarily a principal orbit then it has
\textit{constant principal curvatures} \cite {HOT}, i.e. the shape operator $A_{\xi (t)}$ has constant eigenvalues for
any parallel normal field $\xi (t)$ along any curve. \medskip

Let $M$ be a Euclidean submanifold. The so-called \textit{normal holonomy group} $\Phi ^\perp _p$ of $M$ at $p$ is the
holonomy group of the normal connection of $M$ at $p$. The \textit{normal holonomy theorem} \cite {O1} states that
the connected component of the normal holonomy group acts on the normal space, up to its fixed set, as an
$s$-representation. Any $s$-representation acts \textit{polarly} on the ambient space, i.e. there exists a subspace
$\Sigma$ that meets all orbits  in an orthogonal way \cite {PT} (such a $\Sigma$ is the normal space of any
principal orbit). Conversely, given a polar representation there exists an $s$-representation with the same
orbits \cite{D}. \smallskip

A local group of isometries $G$ of a Riemannian manifold $X$ is said to act \textit{locally polarly} if the distribution
of normal spaces to maximal dimensional (local) orbits is integrable (or, equivalently, autoparallel; see \cite
{PT}). If $G$ acts locally polarly on $X$ and $S\subset X$ is a locally $G$-invariant submanifold, then  the
restriction of $G$ to $S$ acts locally polarly on $S$ (this follows form Corollary 3.2.5 and Proposition 3.2 of
\cite {BCO}, though we will only need the special cases given by Proposition 3.2.9 of this reference and Lemma
2.6 in \cite {OW}).

\smallskip

The Normal Holonomy Theorem was extended to Riemannian submanifolds of the Lorentz space \cite {OW}. The
conclusion, in this case, is that the normal holonomy acts polarly on the (Lorentzian type) normal space. This
means that the normal spaces to  any  maximal dimensional time-like orbit meet any nearby orbit orthogonally. We
will need to make use of this result, though we are only interested in Euclidean submanifolds.

\subsection{Complex submanifolds of $\C P^n$}\label{4.1}

Recall that $\C P^n$ is obtained by $\C^{n+1}\backslash \{0\}$ identifying complex lines through the origin. Hence there is a canonical projection  $\pi : \C^{n+1}\backslash\{0\} \rightarrow \C P^n$.
Of course, one may also regard $\C P^n$ as a quotient of the unit $(2n+1)$-sphere in $\C^{n+1}$ under the action of $U(1)$, i.e., $\C P^n = S^{2n+1}/U(1)$ (this is because every line in $\C^{n+1}$ intersects the unit sphere in a circle; for $n=1$ this construction yields the classical Hopf bundle). Thus one has a submersion $S^{2n+1}\to \C P^n$. The Fubini-Study metric $\langle\, ,\, \rangle_{FS} $ is then characterized by requiring this submersion to be Riemannian.

Let $M \subset \C P^n$ be a full complex submanifold of the complex projective space.

Let us denote by $\widetilde{M}$ the lift of $M$ to $\C^{n+1}\backslash\{ 0\}$, i.e. $\widetilde{M} := \pi^{-1}(M)$. Let ${\mathcal{V}}$ be the vertical distribution of the submersion $\pi: \widetilde{M} \rightarrow M$. It is standard to show that ${\mathcal{V}}\subset {\mathcal{N}}^{\widetilde{M}}$. If $X$ is a tangent vector to $M$ we will write $\widetilde{X}$ for its horizontal lift to $\C^{n+1}\backslash\{0\}$.

The submersion $\pi: \C^{n+1}\backslash\{0\} \rightarrow \C P^n$ is not Riemannian. Anyway, the following O'Neill's type formula holds

\begin{proposition}[O'Neill's type formula]\label{lift} Let $\widetilde{X}, \widetilde{Y} \in \Gamma(\C^{n+1} \backslash \{0\})$ be the horizontal lift of the vector fields $X,Y \in \Gamma(\C P^n)$.
Then,
\begin{equation} \label{fun} (D_{\widetilde{X}}\widetilde{Y})_{\widetilde{p}} =  (\widetilde{{\nabla^{FS}_XY}})_{\widetilde{p}} + {\mathcal{O}}(\widetilde{X},\widetilde{Y})
\end{equation}
where ${\mathcal{O}}(\widetilde{X},\widetilde{Y}) \in {\mathcal V}_{\widetilde{p}}$ is vertical.
\end{proposition}

The  proof is the same as the standard one \cite{O'N}.
Indeed, the restriction $d \pi : {\mathcal V}^{\perp} \rightarrow T \C P^n$ is a dilatation, i.e. $\pi^* \langle\, ,\, \rangle_{FS} = \lambda^2 \langle\, ,\, \rangle$, where $\pi^*$ is the pullback to the horizontal part. An important remark is that the function $\lambda$ is constant on horizontal curves. Hence moving along horizontal curves one remains in the same sphere $S^{2n+1}$ (of radius $1/\lambda$).

\subsection{Parallel normal fields} \label{1.2}
Let $M\subset \RR^n$ be a submanifold and let $\xi$ be a non-umbilic parallel normal field to $M$.
Since we are working locally, we may assume that the different eigenvalue functions of the shape
operator $A_\xi$, $\lambda _1, \cdots, \lambda_g :M\to \RR$  have constant multiplicities and so they are
differentiable functions. Let $E_1\cdots , E_g$ be their associated smooth eigendistributions, i.e., $TM = E_1
\oplus \cdots \oplus E_g$ and $A_{\xi \, |E_i} = \lambda _i \Id _{|E_i}$. The eigendistributions $E_1, \cdots ,
E_g$ are integrable, due to the Codazzi identity. Moreover, each eigendistribution is $A$-invariant.  Indeed,
since $\nabla ^\perp \xi = 0$, $\langle R^\perp _{X,Y}\xi , \eta \rangle = 0$ and so, by the Ricci identity,
$[A_\xi, A_\eta] =0$, for all $\eta$ normal field to $M$.

Assume that one of the eigenvalues, let us say $\lambda _1$ is constant. Using again the Codazzi identity, we get that the eigendistribution
$E_1$  is not only integrable but also autoparallel in $M$ (see
\cite {BCO}). Moreover, any (totally geodesic) integral manifold $S_1(x)$ of $E_1$ is not a full submanifold of
$\RR^n$. Indeed,
$$S(x) \subset  x + T_xS(x) \oplus \nu_x M\, .$$
If the submanifold $M \subset \RR^n$ has flat normal bundle, then all the shape operators
commute and so they can be simultaneously diagonalized. Around a generic point,  there are (unique, up to order)
normal fields $\eta _1 , \cdots , \eta _g$, the so-called \textit{curvature normals}, and $A$-invariant eigendistributions
$E_1, \cdots , E_g$ such that
$$TM = E_1 \oplus \cdots \oplus E_g$$
and $A_{\xi|\,E_i} = \langle \xi , \eta_i\rangle \Id _{E_i}$, for all normal sections $\xi$.

The integral manifolds of $E_i$ are umbilical submanifolds of the ambient space (if $\dim E_i \geq 2$). If a curvature
normal $\eta _i$ is parallel (in the normal connection), then $E_i$ is an autoparallel distribution of $M$.
Moreover, any leaf $S_i(q)$ of $E_i$ is (an open subset of) a sphere, which is totally geodesic in $M$. In fact,
$S_i(q)$ is the sphere of  the affine subspace $q + E_i(q)\oplus \eta_i (q)$ centered at $q + ||\eta_i||^{-2}\eta_i(q)$

\smallskip

Let now $M\subset N \subset \RR^n$ be submanifolds with flat normal bundle and such that $M$ is
$A^N$-invariant. Observe that $\nu N_{|\, M}$ is a parallel (and flat) subbundle of $\nu M$. We relate the curvature normals in  $N$ with the ones in $M$:

\begin{lemma}\label{curvnorm}
Let $M\subset N \subset \RR^n$ be submanifolds with flat normal bundle and such that $M$ is
$A^N$-invariant. Assume that $\eta$
is a parallel curvature normal of $N$ with associated autoparallel eigendistribution  $E$.
Suppose $\bar E := E_{|\, M}$ is contained in $TM$. Then $\bar \eta := \eta _{|\,M}$ is a parallel curvature normal
with associated (autoparallel)  eigendistribution $\bar E$.
\end{lemma}

\begin{proof} Let $\xi$ be a parallel normal field to $M$
which lies in $\nu N_{|\, M}$. Then $A ^N _{\xi\, |\, TM} = A_\xi ^M$ and so $ A ^M _{\xi \, |\, \bar E} =
\langle \xi , \bar \eta \rangle \Id_{\bar E}$.
\newline Let now $\zeta$ be a parallel normal field to $M$ which is tangent to $N$. Since $A^M_\zeta$ commutes
with all shape operators $A^M_\xi$,  it  commutes, in particular, with all $A^M _\xi$ such that  $\xi$ lies in
$\nu N_{|\, M}$. Thus $A^M_\zeta$ has to leave the common eigenspace $\bar E$ invariant. Let us compute
$A^M_{\zeta \, |\,\bar E}$. Let $X, Y$ be  tangent fields to $N$ which lie in $E$. Then $$\langle A^M_\zeta(X),
Y\rangle  = - \langle \nabla ^N _X \zeta , Y \rangle = \langle \zeta , \nabla ^N _X Y \rangle = 0$$
 since $E$ is
autoparallel. Then $A^M_{\zeta \, |\, E}=0 =  \langle \zeta , \bar \eta \rangle \Id_E$. This shows that $\bar
\eta $ is a (parallel) curvature normal of $M$ with associated eigendistribution $\bar E$.
\end{proof}

The same is true if $M,N$ are Riemannian submanifolds of Lorentz space.

\subsection{Parallel and focal manifolds}\label{1.3}

 Let $M \subset \RR^n$ be a  submanifold and let $\xi \neq 0$ be a
parallel normal field to $M$. Observe that this implies that $\nu _0M$ is a non trivial subbundle of $\nu M$.

\smallskip

Assume that $1$ is not an eigenvalue of $A_{\xi(x)}$, for any $x\in M$. The \textit{parallel manifold} is defined by
$$M_\xi := \{ x + \xi (x): x \in M\}$$
and is a submanifold of $\RR^n$.
Note that the normal spaces $\nu_p M$ and $\nu_{p+\xi (p)} M_\xi$ identify since they are parallel (affine) spaces in $\RR^n$.

\medskip

If $1$ is a constant eigenvalue of $A_\xi$ with constant multiplicity then $M_\xi$ is also a submanifold of
Euclidean space, a so-called \textit{focal (or parallel focal) manifold to $M$}, and $$\pi : M \to M_\xi$$ is a submersion, where $\pi (x)= x + \xi
(x)$  (but not in general a Riemannian submersion).
The fibers $\pi ^{-1}(\{\pi (x)\})$ are
totally geodesic in $M$ and $A$-invariant. Indeed, these fibers are the integral manifolds of the
eigendistribution
$$\ve^\pi := \ker  (\Id - A_\xi )$$ associated to the constant eigenvalue $1$.

There is an orthogonal decomposition
$$T_xM = T_{\pi(x)}(M_\xi) \oplus \ve^\pi _x$$
and,  by what remarked in the previous subsection,
$$\pi ^{-1} (\{ \pi (x)\}) \subset \pi (x) + \nu _{\pi (x)} (M_\xi)$$
that is, the fibers lie in the normal space of the focal manifold.

From the last two relations, there is  the orthogonal splitting
$$ \nu _{\pi (x)}(M_\xi) = \nu _x M \oplus \ve^\pi_x\, .$$

The horizontal distribution $\Cal H ^\pi$ of $M$ is the one perpendicular to $\ve^\pi$. Observe that the
horizontal distribution $\Cal H ^\pi$ is not in general integrable but it is $A^M$-invariant, since $\ve^\pi$
is so. By the above relations one has, as subspaces,
$$\Cal H ^\pi _x = T_{\pi (x)}(M_\xi)\, .$$

\subsection{Parallel transport and shape operators of parallel (focal) manifolds}
\label{sollev}

The following discussion is similar to that in \cite{HOT}.
Let  $c(t)$ be a curve in $M_\xi$ and let $q \in \pi ^{-1}(\{c(0)\})$. Then locally there is a unique horizontal lift $\tilde c(t)$ of $c(t)$ with $\tilde c (0) = q$, i.e. $\pi \circ \tilde c = c$
and $\tilde c'(t) \in \Cal H^\pi _{\tilde c (t)}$. Then, $\eta (t) = \tilde c (t) - c(t)$ is a parallel normal
field to $M_\xi$ along $c(t)$, since its Euclidean derivative at $t$ lies in $T_{c(t)}(M_\xi)$. Conversely, if
$\eta (t)$ is the parallel normal field along $c(t)$, with $\eta (0) = q - c(0)$, then $\tilde c (t): = c(t) +
\eta (t)$ is the horizontal lift of $c(t)$ with initial condition $q$. This implies the important fact that the $\nabla
^\perp$-parallel transport along a curve $c$ in $M_\xi$, joining $p$ and $q$, $\tau _c^\perp : \nu _{p}(M_\xi)
\to \nu _{q}(M_\xi)$, maps (locally) $\pi ^{-1}(\{p\})$ into $\pi ^{-1}(\{q\})$.

\medskip

Let now $\tilde \beta (t)$ be a horizontal curve in $M$ and let $\beta (t) = \pi (\tilde \beta (t))$.
Let  $\eta (t)$ be a parallel normal field to $M$ along $\tilde \beta (t)$. Since $\nu _{\tilde \beta (t)}M
\subset \nu _{ \beta (t)}(M_\xi)$, then $\eta (t)$ may  also be regarded as a normal field to $M_\xi$ along
$\beta (t)$. Moreover,  $\eta (t)$ is also a parallel normal field to $M_\xi$ along $\beta (t)$. Indeed, since
$\eta (t)$ is parallel along $\tilde \beta (t)$, one has that
\begin{equation}\label{I}
\frac {d}{dt}\eta (t) = -A^M _{\eta (t)}.\tilde
\beta ' (t) \subset \Cal H ^\pi _{\tilde \beta (t)} = T_{\beta (t)}(M_\xi)
\end{equation}
by the $A^M$ invariance of $\Cal H
^\pi$. \medskip

Using (\ref{I}), since $\beta (t) = \pi (\tilde \beta (t)) = \tilde \beta (t) + \xi
(\tilde \beta (t))$ one has that
$$
\beta '(t) = \tilde \beta '(t) - A^M_{\xi (\tilde \beta (t))}. \tilde \beta
'(t) = (\Id - A^M_{\xi (\tilde \beta (t))}).\tilde \beta '(t).
$$
On the other hand, since $\eta (t)$ is a parallel normal field along
$\beta (t)$,
\begin{equation}\label{II} \frac {d}{dt}\eta (t) = -A^{M_\xi}_{\eta (t)}. \beta '(t)\, .
\end{equation}
Then,  since the expressions (\ref{I}) and (\ref{II}) coincide, and $\tilde \beta (t)$ is an arbitrary horizontal curve, one
gets the well-known formulae relating the shape operators of $M$ and $M_{\xi}$, sometimes called ``\textit{tube formulae}'' \cite{BCO}
\begin{equation}\label{tubo1} A^{M_\xi}_{\eta _x} = A^M_{\eta _x}  (\Id - A^M_{\xi (x)})^{-1}_{|\Cal H ^\pi _x}
\end{equation}
for all $\eta _x \in \nu _xM$

In a similar way we have
\begin{equation}\label{tubo2}A^{M}_{\eta _x \, |\Cal H^\pi} = A^{M_\xi}_{\eta _x}  (\Id - A^{M_\xi}_{-\xi (x)})^{-1}
\end{equation}
for all $\eta _x \in \nu _xM$.

\subsection{Parallel manifolds at infinity}\label{1.4}

Let $M\subset \RR^n$ be a submanifold with a parallel normal field $\xi$. In some cases, for our geometric
study of $M$, there are not enough parallel manifolds $M_{\lambda \xi}$ to $M$ in the Euclidean space, $\lambda
\in \RR$. So, it is convenient to regard $M$ as a Riemannian submanifold of a Lorentz space, in which case
the family of parallel manifolds to $M$ is enlarged. This construction is worth while when $0$ is an
eigenvalue of    $A_\xi$ with constant multiplicities (otherwise, everything can be carried out in the original Euclidean space). In this case the integral manifolds of the $A$-invariant
autoparallel distribution $E = \ker  A_\xi$ are the fibers of the submersion defined by passing to a
parallel focal manifold. Observe that $E$ is not in general the nullity distribution, i.e. the one given by the
intersection of the kernels of all shape operators.

For this purpose, let $L^{n+2} = (\RR^{n+2}, \langle\ ,\ \rangle)$, where
$$\langle x,y\rangle = -x_1y_1 +
x_2y_2 + \cdots + x_{n+2}y_{n+2}\, . $$
 Recall that the hyperbolic space of radius $r$ is given by
$$H^{n+1} (r) = \{ x \in L^{n+2}:\, \langle x, x \rangle = -r^2,\ x_1>0\}\, .$$
In this way $H^{n+1} (r)$ is regarded as a totally umbilical (Riemannian) hypersurface of $L^{n+2}$. Indeed, the vector field $\eta(x)=-x$ is a  parallel (time-like) normal field to $L^{n+2}$ and $A^{L^{n+2}}_\eta = \Id$.
Now regard $\RR^n$ as a horosphere $Q^n$ of the hyperbolic space, which is also a totally umbilical
hypersurface. In this way one can regard $M$ as a submanifold of Lorentz space. Now there is in $M$ an extra,
somewhat trivial, parallel normal field given by the restriction to $M$ of the vector field of the
hyperbolic space which we also call $\eta$ (the normal vector  field to the horosphere, in the hyperbolic space, gives no useful information). Now, in the Lorentz space, we have  the family of parallel manifolds to
$M$ given by $M_{a\xi + b\eta}$, cf. \cite{DG}.

There is no essentially new parallel manifold except the focal one $M_{\tilde \xi}$, where $\tilde \xi = a\xi
+\eta$,  $a$ is small enough so that $\tilde \xi$ be time-like (it is convenient, for reasons related to the
normal holonomy of the focal manifold, to choose a time-like parallel normal field for the focalization).

In this way $$\pi : M \to M_{\tilde \xi}\, ,$$ where $\pi (x)= x + \tilde \xi (x)$ is a submersion. The fibers $\pi
^{-1}(\{\pi (x)\})$ are just the integral manifolds of $E = \ve^\pi = \ker  A_\xi = \ker (\Id -
\tilde A_{\tilde \xi})$, where $\tilde A$ denotes the shape operator of $M$ as a submanifold of the Lorentz
space. Note that the focal manifold $M_{\tilde \xi}$ is contained in a de Sitter space of radius $\vert a \vert  \Vert \xi \Vert$.

One can relate parallel transport and shape operators in parallel (focal) manifolds like in the previous subsection (see \cite{OW, BCO}).

\subsection{Holonomy tubes}\label{1.5}

Let $M$ be a  Euclidean submanifold or a Riemannian  submanifold of the Lorentz space. Let $\xi _p \in \nu _pM$. If
$M$ is a submanifold of Lorentz space then $\xi_p$ is assumed to be time-like. The \textit{holonomy tube} around $M$
through $\xi _p$ is defined by
$$M_{\xi _p} = \{ c(1) + \xi (1) \} =  \{ c(1) + \tau _c^\perp (\xi _p) \}\, ,$$
where $c:[0,1] \to M$ is an arbitrary curve starting at $p$ and $\xi (t)$ is the parallel transport of $\xi _p$
along $c(t)$. If $1$ is not an eigenvalue of the shape operator $A_{\xi _p}$,  then $M_{\xi _p}$ is a
submanifold of the ambient space, e.g. if $\xi _p$ is near $0$ (perhaps by making $M$ smaller). One has a
projection $\pi : M_{\xi _p} \to M$, defined by $\pi (c(1) + \xi (1)) = c(1)$. Moreover, $q \mapsto \eta (q): =
\pi (q) -q$ is a parallel normal field to $M_{\xi _p}$ and so we have that $M$ is a parallel manifold (in general, focal)  to its
holonomy tube. Namely,
$$ M = (M_{\xi _p})_\eta$$
Observe that $\eta (p+ \xi _p) = -\xi _p$. Note that the fibers of $\pi$ are given by the orbits of the normal
holonomy group of $M$. Namely,
$$\pi ^{-1}(\{\pi (p)\}) =  \pi (p) + \Phi ^\perp _{\pi (p)}. (p- \pi (p))$$

In the Lorentzian case this fiber is contained in the hyperbolic space of  the normal space given by the
time-like vector $\xi _p$, which is invariant under this holonomy action. Moreover, this action, restricted to
this hyperbolic space  is locally polar. Observe that in this Lorentzian case, the holonomy tube is a Riemannian
submanifold, since $M$ and the holonomy orbit $\Phi ^\perp _{\pi (p)}. (p- \pi (p))$ are so. \medskip

 On the other hand, let  $M$ be a submanifold with a parallel normal field $\eta$ such that $1$ is a constant eigenvalue,   with constant multiplicity $r$, of $A_{\eta}$, $r < \dim (M)$. By \S~\ref{sollev}  (see also \S~\ref{1.4} for the Lorentzian case), we have that
$$(M_\eta)_{-\eta (q)} \subset M$$
for all $q\in M$. That is, $M$ is foliated by the holonomy tubes around the parallel manifold $M_\eta$ (this
foliation could be trivial, i.e., with only one leaf).

\smallskip

Let us observe that if the normal vector $\xi _p \in \nu _pM$ extends to a parallel normal field
then the holonomy tube $M_{\xi _p}$ is a parallel non-focal manifold to $M$. This is equivalent to the fact that
$\xi _p$ is fixed by the normal holonomy group of $M$.

\medskip

If the orbit $\Phi ^\perp _{\pi (p)}. (p- \pi (p))$ is maximal dimensional (and hence isoparametric
in the normal space) then the holonomy tube $M_{\xi _p}$ has flat normal bundle; see \cite {HOT} for the
Euclidean case. The Lorentzian case is similar since normal holonomy orbits, through principal time-like
vectors, are isoparametric in a hyperbolic space (and also when regarded as Riemannian submanifolds of the
normal space).

Conversely, if the holonomy tube $M_{\xi_p}$ has flat normal bundle then the holonomy orbit must have flat normal
bundle, in the normal space, and hence is maximal dimensional. In the Euclidean space this is well-known, since
singular orbits of  $s$-representation must have non-trivial normal holonomy \cite {HO}. In the Lorentzian space
the polar actions are, orbit-like,  essentially the same as in Euclidean space, up to some transitive factors in
hyperbolic space or horospheres (see \cite[Theorem 2.3]{OW}).

\smallskip

We shall call \textit{full holonomy tube} a holonomy tube with flat normal bundle. By the above discussion we have

\begin{lemma}\label{full_hol-tube}
The holonomy tube $M_{\xi_p}$ has flat normal bundle, i.e., it is a full holonomy tube if and only if the normal holonomy orbit $\Phi ^\perp _{\pi (p)}. (p- \pi (p))$,
with $\pi : M_{\xi _p} \to M$ the projection, is maximal dimensional (and hence an isoparametric submanifold of the normal space $p + \nu _pM$).
\end{lemma}

\begin{remark}\label{tubolonomico}
Let $M_{\xi_p}$ be a full holonomy tube and let $\bar \eta$ be a  curvature
normal, with associated autoparallel eigendistribution $E$ of the isoparametric submanifold $p + \Phi ^\perp
_{\pi (p)}. (p- \pi (p))$ of the normal space $p + \nu _pM$ ($\bar \eta $ must be parallel in the normal
connection). Then $\bar \eta$ is the restriction to $ p + \Phi ^\perp _{\pi (p)}. (p- \pi (p))$ of a parallel
curvature normal $\eta$ of $M_{\xi _ p}$, whose associated eigendistribution, restricted to the holonomy orbit,
coincides with $E$ (cf. Lemma~\ref{curvnorm}, \S~\ref{1.2}, page~\pageref{curvnorm}). Moreover, the restriction of $\eta$ to  any holonomy orbit is a curvature normal of this
orbit.
\end{remark}

\section{Foliation by holonomy tubes}\label{2}

This section is the main core of this paper.
We begin with a submanifold $M \subseteq \RR^n$ endowed with a parallel normal field $\xi$ such that 1 is an eigenvalue with constant multiplicity of $A_\xi$. As we know from the previous subsection, $M$ is foliated by the holonomy
tubes $H(x):=(M_\xi)_{x- \pi (x)} = (M_\xi)_{-\xi(x)} $, $x\in M$. We may   assume, since we are working locally that
all these holonomy tubes have the same dimension.

 In \S~\ref{1.6} we describe the properties of this foliation (Proposition~\ref{slice}). It comes out there is a strong similarity with polar actions. Indeed, the orthogonal distribution to $H(x)$ is integrable and its leaves $\Sigma (x)$ behave like sections in a polar representation.
In \S~\ref{2.1} we compare the eigendistributions of nearby parallel manifolds.

Then we introduce a canonical foliation for submanifolds of $\RR^n$ whose normal holonomy group acts irreducibly and non-transitively on the unit sphere of the normal space.
In \S~\ref{2.2} we begin with a submanifold $N \subset \RR^n$, take a full holonomy tube $N_{\zeta_p}=:M$ and we assume there is a parallel normal field $\xi$ on $M$ with $\ker A_\xi \neq \{ 0\}$. Then $M$ is foliated by the
holonomy tubes $H^\xi (x)$ around the focal manifold  at infinity  $M_{\tilde \xi}\subset L^{n+2}$. This may seem to depend on the choice of the parallel normal field $\xi$, but in \S~\ref{indep} we show it is not the case. Now, $N$ can be regarded as a focal manifold of $M$, with projection $\pi: M \to N$. In \S~\ref{down} we project down to $N$ the canonical foliation on $M$. The homogeneity of this canonical foliation is finally proven in \S~\ref{homog} provided that the horizontal distribution of a full holonomy tube is covered by kernels of shape operators (Theorem~\ref{main-lemma}).

\subsection{Polar-like properties of the foliation by holonomy tubes} \label{1.6}

 Let $M$ be submanifold of Euclidean space or,
more generally,  a Riemannian submanifold of Lorentzian space. Let $\xi$ be a parallel normal field to $M$ and
assume that $M_\xi$ is a parallel focal manifold to $M$, i.e $1$ is an eigenvalue, with constant multiplicity,
of the shape operator $A_\xi$. As we have observed in the previous subsection, $M$ is foliated by the holonomy
tubes $(M_\xi)_{x- \pi (x)} = (M_\xi)_{-\xi(x)} $, $x\in M$, that we assume are all of the same dimension.
%

Let $\tilde \nu$ be the distribution in $M$ which is perpendicular to the tangent spaces of the holonomy tubes.
Observe that the restriction of $\tilde \nu$ to any fiber $S(x) = \pi^{-1}(\pi (x))$ coincides with the
distribution given by the normal spaces to the orbits of the normal holonomy group $\Phi^\perp_{\pi (x)}$ in
$S(x)$.  But this action must be locally polar (see the end of \S~\ref{1.1}). Then the normal spaces to
the orbits is an autoparallel distribution. This shows that $\tilde \nu$ is autoparallel, since the fibers
$S(x)$ are totally geodesic.

Let us examine the construction of  the integral manifolds $\Sigma (q)$ of $\tilde \nu$ more closely. This construction is implicit in the proof of [Lemma 2.6]{OW}).  The main point is that the restriction, to an invariant submanifold, of a locally polar action is again locally polar [Lemma 2.6]{OW}. Indeed,
\begin{equation}\label{(*)}
\Sigma (q) =  S(q) \  \cap \   q + \nu _{-\xi (q)}\big( \, \Phi ^\perp_{\pi (q)}.(-\xi (q)\, )\, \big) \end{equation}
 where $\Phi ^\perp $ denotes the normal holonomy group of $M_\xi$
and the normal space to the  holonomy orbit is inside $\nu _{\pi (q)}(M_\xi)$. Observe that the above expression
is independent of $x$ in a given $\Sigma (q)$,  and shows that $\Sigma (q)$ is totally geodesic.

If $x\in \Sigma (q)$ then, by (\ref{(*)}), $(x-q)$ belongs to the normal space, in $\nu _{\pi (q)}(M_\xi)$, of the
holonomy orbit $\Phi ^\perp_{\pi (q)}.(-\xi (q)\, )$. This  orbit has the same dimension as its nearby orbit
$\Phi ^\perp_{\pi (q)}.(-\xi (x)\, )$ (note that $\pi (x) = \pi (q)$). This implies that $(x-q)$ is a fixed
vector of the slice representation of the isotropy subgroup $(\Phi ^\perp_{\pi (q)})_{-\xi (q)}$ of $\Phi ^\perp_{\pi (q)}$ at $-\xi(q)$. Hence the extension $\eta$ of $(x-q)$  to a $\Phi ^\perp_{\pi (q)}$-invariant normal field to $\Phi ^\perp_{\pi (q)} (-\xi (q))$, where $\eta (q) = x-q$,  is parallel in the
normal connection of the orbit, regarded as a  submanifold of $\nu _{\pi (q)}(M_\xi)$ (see Proposition 2.4 of
\cite {OW} and Proposition 3.2.4 of \cite {BCO}). Observe that this orbit could be non-principal in the ambient
space. Since $x \in \Sigma (q)$ is arbitrary we obtain that
$$ -q + \Sigma (q) \subset \nu _0 \big( \Phi ^\perp
_{\pi (q)}  (-\xi (q))\big )\, $$
(recall that $\nu _0$ is the maximal parallel and flat subbundle of $\nu$).
By the above construction we have that the normal parallel transport, along any curve in $\Phi ^\perp_{\pi (q)} (-\xi
(q))$, from $q$ to $q'$, maps $-q + \Sigma (q)$ into $-q' + \Sigma (q')$. It is standard to prove and well-known
(see \cite {HOT}) that a parallel and $\Phi ^\perp_{\pi (q)}$-invariant normal field to the  holonomy orbit $\Phi
^\perp_{\pi (q)} (-\xi (q))$ extends to parallel normal field $\eta$ of the holonomy tube $H(q): =
(M_\xi)_{-\xi(q)}$ (we require that $\eta$ be both parallel and $\Phi ^\perp_{\pi (q)} (\-\xi (q))$-invariant
since the holonomy orbit could be non-full). Then, $$ -q + \Sigma (q) \subset \nu _0 \left( H(q) \right)\, .$$
Moreover,  the sets $-x + \Sigma (x)$ move parallel with respect to the normal connection of $H(q)$, $x\in H(q)$. This
implies that its tangent spaces give rise to a parallel and flat subbundle of the normal bundle $\nu
\big(H(q)\big)$ in the ambient space. That is, the restriction to $H(q)$  of $\tilde \nu$ is a parallel and flat
subbundle of $\nu \big(H(q)\big)$.

Let  $x\in \Sigma (q)$ and  let  $\eta $ be the  parallel  normal field to $H(q)$ with $\eta (q) = x-q$. Then
observe that $H(x) = (H(q))_\eta$, i.e. the different holonomy tubes inside $M$ are parallel manifolds.

We can now prove that $\tilde \nu$ is a $A^M$-invariant distribution of $M$. Indeed, let $X, \, Y$ be vector fields on
$M$, where $X$ is tangent to the holonomy tubes and $Y$ is perpendicular, i.e. $Y$ lies in $\tilde \nu$. The
 Euclidean derivative $(\nabla ^E _X Y)_x \in T_xH(x)\oplus \tilde \nu _x$, since $\tilde \nu _{|\, H(x)}$ is a
 parallel subbundle of the normal bundle in the ambient space. Then it has no normal component to $M$.
 Thus $\alpha ^M (\tilde \nu ,\tilde \nu ^\perp) =0$ and therefore $\tilde \nu$ is $A^M$-invariant.

We summarize what we have proven in the following:

\begin{proposition}\label{slice}
Let $M$ be a Euclidean submanifold or, more generally, a Riemannian submanifold of Lorentz
space. Let $\xi$ be a parallel normal field to $M$, with a constant eigenvalue $1$ with constant multiplicity.
For any $x\in M$, we denote by $H(x) \subset M$ the  holonomy tube $(M_\xi)_{-\xi (x)}$ of the focal manifold $M_\xi$ and we assume that all $H(x)$ have the same dimension.
Let $\tilde \nu$ be the distribution in $M$ which is
perpendicular to the family of holonomy tubes. Then,
\begin{enumerate}
\item [(i)] $\tilde \nu$ is
autoparallel and invariant under all shape operator of $M$. Moreover, if $\Sigma (x)$ is a leaf of $\tilde \nu$
through $x$, then $$\Sigma (x) = (x+\nu_x H(x)) \cap M\, .$$
\item [(ii)] The leaves  $\Sigma (q)$ are invariant under the parallel transport in the normal bundle of the
focal manifold $M_\xi$. That is, if $c$ is a curve in $M_\xi$ from $\pi (x)$ to $\pi (y)$ then
$$\tau ^\perp _c (\Sigma(x)) = \Sigma(y)\, .$$
\item [(iii)] The restriction of $\tilde \nu$ to any $H(x)$ is a parallel (and flat) subbundle of $\nu _0H(x)$.
Moreover,
$$ \Sigma (x) \subset x + (\nu _0 H(x))_x$$ and  $\Sigma (y)$ moves parallel, in the normal connection of the holonomy
tube $H(x)$. That is, if $c$ is a curve in $H(x)$ from $y$ to $z$, then
$$\tau ^\perp _c (\Sigma(y)) = \Sigma(z)\, .$$
\item [(iv)] Let $x\in \Sigma (q)$ and  identify   $(x-q) $ with  the  parallel  normal field to $H(q)$ with this
initial condition at $q$. Then  $H(x) = (H(q))_{{x-q}}$.
\end{enumerate}
\end{proposition}

\subsection{Nearby parallel manifolds}\label{2.1}

Let $M \subset \RR^n$ be a submanifold with a parallel normal field $\xi$ such that the eigenvalues of the
shape operator $A_\xi$ have constant multiplicities. Let $0,\,  \lambda _1 < \cdots, < \lambda _g$ be the
different eigenvalue functions with associated eigendistributions $E_0 ,  \cdots , E_g$.  Let $\eta$ be another
parallel normal field to $M$, such that $1$ is not an eigenvalue of $A_\eta$. Consider  the non-focal parallel
manifold $M_\eta$.

Assume that the eigenvalue functions  of $A^{M_\eta }_\xi$ are $0, \, \bar \lambda _1 <
\cdots, < \bar \lambda _g$, with associated eigendistributions $\bar E_0, \cdots , \bar E_g$, where $\text
{dim}(E_i) = \text {dim}(\bar E_i)$, for all $i= 0, \cdots , g$ (and we assume that the same is true if we
re-scale $\eta$ by a real number $0\leq t \leq 1$). Note that $A_\eta$ must leave  the eigendistributions $E_0 , \cdots ,
E_g$ invariant, since it commutes with $A_\xi$. Thus, in this case, from the tube formulae relating shape operators of
parallel manifolds, $A_\eta$ has only one eigenvalue function, let us say $\beta _i$ in each $E_i$,
$i\geq 1$. Since we are assuming that $\eta $ is small, we have that, for $i=1, \cdots , g$
$$\bar \lambda _i \circ h = \frac {\lambda _i} {1 - \beta _i}\, ,$$
where $h: M\to M_\eta$ is the parallel map, i.e. $h(q)= q + \eta (q)$.

Let $J$ the subset of $\RR$ which consists of the constant eigenvalues of $A_\xi$ and let $J^\eta$ the
analogous subset with respect to $A^{M_\eta}_\xi$. Let now $a \in J\cap J^\eta$. Then, possibly  by re-scaling
$\eta$ (actually, we are assuming  that $a$ belongs to $\cap _tJ^{t\eta}))$ we must have that there is an index
$j$ such that $\lambda _j = \bar \lambda _j \circ h \equiv a$. This implies that $\beta _j \equiv 0$.

\noindent Let now $$I = \{ i :\,  \lambda _i = \bar \lambda _i \circ h \text { and it is constant}\}$$ and $$E =
\underset {j\in I} {\oplus}  E_j$$

Observe that $A_{\eta \, |E} \equiv 0$. So, if $c(t)$ is a curve that lie in $E$, then $\frac {d}{dt}\eta (c(t))
\equiv 0$, since both tangential and normal part of the derivative vanish. So we have proved the following

\begin{lemma}\label{subset} The parallel normal field $\eta $ is constant along $E$.
\end{lemma}

\subsection{The canonical foliation of a full holonomy tube}\label{2.2}

Let $N\subset \RR^n$ be a submanifold of Euclidean space and let  us consider a holonomy tube
$$ M= N_{\zeta _q} $$
 around $N$, where $\zeta _q$ is a principal vector for the
normal holonomy   group $\Phi ^\perp_q$ of $N$ at $q$, and $1$ is not an eigenvalue of  $A^N_{\zeta _q}$. Then
$M$ has flat normal bundle, i.e., it is a full holonomy tube. Let $\pi: M : = N_{\zeta _q} \to N$ be the projection and let $\psi (p) = \pi (p) -p$. Hence, $N = M_\psi$, i.e. the manifold $N$ is a parallel (focal, if $N$ has non-flat
normal bundle) manifold to its holonomy tube.

Let, for $p \in M$, $$S(p) = \pi ^{-1}(\pi (p)) = p + \Phi ^\perp_{\pi (p)}. (p - \pi (p)) \, .$$

Since we are working locally, we may assume that $N$ is simply connected, so its normal holonomy group (and
hence $S(p)$)  is connected.

 For a generic $p \in M$, the common eigenspaces of the shape operators of  $M$ define, in a neighbourhood
 $U$ of $p$,  $C^\infty$ eigendistributions
   $E'_1, \cdots , E'_s$
of $M$, with associated  $C^\infty$ curvature normals  $\eta _1, \cdots , \eta _s$. We assume $U=M$. Observe that $\ker  (
\Id -A^M_\psi) = \ve$, where $\ve$ is the vertical distribution of $M$, i.e. $\ve_x = T_xS(x)$. We have that
$\ve$ is the direct sum of some of the eigendistributions $E'_1, \cdots , E'_s$ (see the remark at page~\pageref{tubolonomico} in \S~\ref{1.5}). Namely, those eigendistributions whose index $i$ verify that $\langle \psi , \eta _i\rangle \equiv 1$. We
may assume that
$$ \ve = E'_1 \oplus \cdots \oplus E'_l
$$ where $l < s$. Observe that the curvature normals $\eta _1, \cdots ,  \eta _\ell$ are parallel, since they are the
extensions of the curvature normals of any fiber $S(x)$, which is an  isoparametric submanifold (see the remark at page ~\pageref{tubolonomico} in \S~\ref{1.5}).

Let   $\xi \neq 0$ be a parallel normal field to $M$ and let, $\ker (A^M_\xi) = E^\xi _0,\  E^\xi _1 ,
\cdots , E^\xi _r$ be the eigendistributions associated to the constant eigenvalues of the shape operator
$A^M_\xi$ (we may have to consider a smaller $M$).  Observe that any $E^\xi _k$ is the sum of some of the
eigendistributions $ E'_1 , \cdots , E'_s $.   Since $\langle \xi , \eta _k \rangle$  is a constant eigenvalue of $A^M_\xi$ (for $0\leq k \leq \ell$) we get  that, if $1\leq i \leq \ell$, then $$E'_i \subset  E^\xi _{j(i)}$$  for some $j(i) \in \{ 1, \cdots , r\}$.
 Observe that in general it could be that $\langle \xi, \eta _i\rangle \equiv  \langle \xi, \eta _j\rangle$ for
 $i\neq j$ and so, in this case,  $E_i\oplus E_j$ is contained in some eigendistribution of $A^M_\xi$.

\begin{assumption} In the sequel of this section we will suppose that the normal holonomy group $\Phi^\perp$ of $N$ acts irreducibly and not transitively on the
 normal space.
\end{assumption}

Therefore \textit{$S(x)$ is an irreducible isoparametric submanifold of the normal space
 $\nu _{\pi (x)} N$} (observe that $N$ must be an irreducible and full submanifold).

 \medskip

Now, \textit{assume further that $0$ is a constant eigenvalue of $A^M_\xi$,i.e. $E^\xi_0$ is non-trivial}. We introduce a \textit{canonical foliation} of $M$, starting from $\xi$, but we will later show it is independent on $\xi$ (\S~\ref{indep}). Recall  $M$ is foliated by the
holonomy tubes $H^\xi (x)$
 around the focal manifold $M_{\tilde \xi}\subset L^{n+2}$, that we assume (possibly in a smaller $M$) are
all of the same dimension (see \S~\ref{1.4} and ~\ref{1.5}). To visualize the holonomy tube let us define the
equivalence relation in $M$,  $x\underset {\xi} {\sim} y$ if there is a curve in $M$ from $x$ to $y$ and
such that it is always perpendicular to the distribution $E^\xi _0$. Then, locally,
$$H^\xi (x) = \{y\in M:\, x\underset {\xi}  {\sim} y\}\, .$$
From the Homogeneous Slice Theorem \cite {HOT} (the local version follows from Theorem 3.1 in \cite {OW}) one
   has that, starting from $x\in M$ and  moving perpendicularly to $E^\xi_0$ one can reach any other point of $S(x)$. Indeed,
   let $\Cal D = \ve \cap E^\xi _0$. One has that $\Cal D \neq \ve$, otherwise $0 = A^M_{\xi  |\ve_x}= A^{S(x)}_{\xi
   (x)}$. Hence the restriction of $\xi$ to $S(x)$ is a parallel normal field whose shape operator is null. Hence
   $S(x)$ is not full in the normal space $\nu _{\pi (x)}M$. A contradiction.
 Now observe that $\Cal D = \ker  (\Id - A^{S(x)}_{\vartheta})$, where $\vartheta = \psi _{|S(x)} - \xi
 _{|S(x)}$. Then, beginning with a point $x \in M$,  moving perpendicularly to $\Cal D$,
 but remaining inside $S(x)$, we reach any other point of $S(x)$, by the Homogeneous Slice Theorem. So, moving
 perpendicularly to $E^\xi _0$, starting at $x$, we reach any point in $S(x)$.

 So we have proven that
\begin{equation}\label{(**)}S(x) \subset H^\xi (x)\, . \end{equation}
 Let now $\tilde \nu ^\xi$ be the normal space to the foliation of $M$ by the holonomy tubes $H^\xi (x)$ and let
 us denote by $\Sigma _\xi (x)$ the totally geodesic leaves of $\tilde \nu ^\xi$. By
 the Proposition~\ref{slice} in \S~\ref{1.6}, if $ x\in \Sigma _\xi (q)$,
$$ H^\xi (x) = (H^\xi (q))_\varsigma \, , $$
 where $\varsigma$ is the parallel normal field to $H^\xi(q)$ with $\varsigma (q) = x - q$.

\begin{remark}\label{rem4}
(i) Observe that $H^\xi (q)$ has
 flat normal bundle. Indeed,  $\nu H^\xi(q) = \tilde \nu ^\xi_{\,|
 H^\xi (q)} \oplus \nu M _{\, |H^\xi (q)}$ and both subbundles are parallel and flat (see  Proposition~\ref{slice} in \S~\ref{1.6}).
\newline
 (ii) By (\ref{(**)}),  the restrictions of the parallel curvature normals $\eta _1 , \cdots , \eta _\ell$ of $M$ to
 any holonomy tube
 $H^\xi (q)$ are parallel curvature normals of this tube. The associated eigendistributions are just the
 restriction to $H^\xi (q)$ of the corresponding eigendistributions $E'_1 , \cdots , E'_\ell$ of $M$ (see Lemma~\ref{curvnorm} in \S~\ref{1.2}).
\end{remark}

We continue with the assumptions before the above remark.

 Since $\tilde \nu ^ \xi \subset E^ \xi _0 = \ker (A^M_\xi)$ we have that $\xi$ is constant along $\Sigma
 _\xi (q)$, in the ambient space. So,
 $$\xi (q) = \xi (x)$$
 as vectors of the ambient space ($x\in\Sigma
 _\xi (q)$). The same is true for any point in $H^\xi (q)$, i.e.
 $$\xi (q') = \xi (q'+ \varsigma (q'))$$
for all $q' \in H^\xi (q)$.

 We can now apply  Lemma~\ref{subset}, since the shape operators
 $A^{(H^\xi(q))_\varsigma }_\xi$ and $A^{H^\xi(q)}_\xi$ share the same constant eigenvalues
$\langle \xi , \eta _1\rangle, \cdots , \langle \xi , \eta _\ell\rangle$. So we conclude that $\varsigma$ must be
constant, in the ambient space, along any fiber $S(q')$, $q'\in H^\xi(q)$. Then, by Proposition~\ref{slice} in \S~\ref{1.6}, we get that the sets $\Sigma _\xi (y)$ are constant (i.e., differ by a translation) if $y$ moves
in  $S(q')$, for all $q'\in H^\xi(q)$ (of course locally).

\subsection{Independence of the foliation on the parallel normal field}\label{indep}

Let us decompose
$$ \tilde \nu ^\xi =    \tilde \nu ^\xi _1 \oplus \cdots \oplus  \tilde \nu ^\xi _t$$
into  different eigendistributions of the family of shape operators of $M$, restricted to $\tilde \nu ^\xi$
(perhaps in smaller $M$). Let $\eta _{h(1)}, \cdots \eta _{h(t)}$ be the associated curvature normals, i.e.
$A^M_{\mu \, |\, \tilde \nu ^\xi _i} =  \langle \mu , \eta_{h(i)}\rangle \Id _{\tilde \nu ^\xi _ i}$. Observe that
the eigendistribution   $E'_{h(i)}$ contains    $\tilde \nu ^\xi _i$ and there is no reason for the equality.
Since $\tilde \nu ^\xi$ is autoparallel and $A^M$-invariant, one has that the restriction $\eta _{h(i) |\,
\Sigma _\xi (y)}$ is a curvature normal of $\Sigma _\xi (y)$, for all $y\in M$.

\medskip

Since the integral manifolds $\Sigma _\xi (y)$ are constant along $S(x)$, $y\in S(x)$, we must have that $\eta
_{h(i)|S(x)}$ is a constant normal vector field to $S(x)$ (regarded as a full submanifold of   $\nu _{\pi
(x)}N$). Then $\eta _{j (i)} = 0$ and  so $\Sigma _\xi (x)$ is totally geodesic in the ambient space for all $x
\in M$ (an hence an open subset of an affine subspace). This shows, since $\tilde \nu ^\xi $ is $A^M$-invariant,
that $\tilde \nu ^\xi$ is contained in the nullity of the second fundamental form $\alpha ^M$. Or equivalently,
$$ \tilde \nu ^\xi \subset \underset {\eta \in \nu M}  {\bigcap}\ker  A _\eta\, .$$
In particular, if $\xi '$ is any other given parallel field to $M$ with $0$ as constant eiegenvalue of $A^M_{\xi'}$  (possibly in a smaller $M$) making the same constructions for $\xi '$, one has  that
$$\tilde \nu ^\xi \subset E^{\xi '}_0 \, .$$
Since the distribution $(\tilde \nu ^\xi)^\perp$ is integrable (the integral manifolds are $H^\xi (x)$) one has
that locally
$$H^{\xi '}  (x) \subset \ H^\xi (x)$$
for all $x\in M$ (recall that $H^{\xi ' }$ is obtained by moving perpendicularly to $E^{\xi ' }$). But in the same way we must have the other
inclusion. So, locally,
$$H^\xi   (x) = \ H^{\xi '}  (x)$$
or equivalently
$$ \tilde \nu ^ \xi  = \tilde \nu ^{\xi '}\, .$$

\begin{remark}\label{horiz}
$\tilde \nu ^\xi$ is horizontal with respect to $\pi$, i.e. $\tilde \nu  ^\xi \subset \ve
^\perp$.  This follows immediately from the fact that $S(x)\subset H^\xi (x)$.
\end{remark}

\subsection{Projecting down the foliation}\label{down}

Observe that $x'-x$ belongs to $\Sigma _\xi (x)$, for all $x'\in \Sigma _\xi (x)$ since this submanifold is
totally geodesic in the ambient space.

In this way $M$ can be locally written as the union of  parallel manifolds to $H^\xi (x)$
$$ M = \underset {x'\in \Sigma _\xi (x)} \bigcup (H^\xi (x))_{x'- x} \text { \  \ \ \  \ \ \ (locally)}$$
where $(x'- x)$ is identified with a parallel normal field along $H^\xi (x)$, with this initial condition at
$x$.

 It is standard to prove, since $\Sigma _\xi (x')$ is locally constant, for $x' \in S(x)$,
that $\tilde \nu ^\xi$ projects down to an autoparallel  distribution $\pi (\tilde \nu ^\xi)$ of $N$ which is
contained in the nullity of the second fundamental form $\alpha ^N$ of $N$. The integral manifolds are $\pi
(\Sigma _\xi (x))$, which are open subsets of affine subspaces of the ambient space. The complementary
distribution is integrable with $A^N$-invariant leaves given by $\pi (H^\xi (x))$. Moreover, the restriction of
$\pi (\tilde \nu ^\xi)$ to $\pi (H^\xi (x))$ is a parallel and flat subbundle of the normal space $\nu (\pi
(H^\xi (x)))$, in the ambient space. Namely, if $x \in \Sigma _\xi (q)$, $x-q$ can be extended to a parallel normal field $\eta$ to $H^\xi (q)$ which we have seen to be constant on $S(q)$.  Then it projects down to a parallel normal field of $\pi ( H^\xi (q))$. We also obtain that
$$ N = \underset {y\in \pi (\Sigma _\xi (x))} \bigcup (\pi (H^\xi (x)))_{y- \pi (x)} \text { \  \ \ \  \ \ \ (locally)}$$

 \medskip

\begin{lemma}\label{hol-tubo}
\begin{enumerate}
\item[(i)] The normal holonomy $\Phi ^\perp _H $ of $\pi (H^\xi (x))$ at $\pi (x)$,  restricted to the invariant subspace $\nu _{\pi (x)} N$, coincides with the normal holonomy group $\Phi ^\perp _N$ of $N$ at $\pi (x)$.
\item[(ii)] $\nu _0 (\pi (H^\xi (x))) = (\pi (\tilde \nu ^\xi))_{\, | \, \pi (H^\xi (x))}$.
\end{enumerate}
\end{lemma}

\begin{proof} The inclusion in part (i) of the first group into the second is clear. Let us prove
the other inclusion. Since the distribution $\pi (\tilde \nu ^\xi)$ of $N$ is inside the nullity of $\alpha ^N$, by the Ricci identity, it is in the nullity of the normal curvature tensor $R^\perp$ of $N$ and
in particular $R^\perp _{X,Y} =0$ if $X$ lies in $\pi (\tilde \nu ^\xi)$ and $Y$ in the perpendicular (integrable)
distribution. Then, if $c$ is a curve in $N$ the parallel transport $\tau ^\perp _c$ coincides with $\tau ^\perp
_{c_2} \circ \tau ^\perp _{c_1}$, where $c_1$ is a curve which lies in $\pi (H^\xi (x))$ and $c_2$ lies in $\pi
(\Sigma _\xi (x))$ (see the lemma in the Appendix of  \cite {O2}). Both curves $c_1$ and $c_2$ are loops, if $c$ is short
(because in our situation we have two integrable distribution). But $\tau ^\perp_{c_2} = \Id$, since the normal
space of $N$ is constant along any curve in the nullity of $\alpha ^N$. This shows the other inclusion.

Part (ii) follows from the fact that
$$ \nu (\pi (H^\xi (x))) = (\pi (\tilde \nu ^\xi))_{\, | \, \pi (H^\xi (x))} \oplus (\nu N)_{\, |\, \pi (H^ \xi
(x))}$$ and that the first subbundle of this sum is parallel and flat (recall that the normal holonomy group of $N$ acts irreducibly).
\end{proof}

\subsection{Homogeneity of the canonical foliation}\label{homog}

 We come back to the foliation $x\mapsto H^\xi (x)$ in the principal holonomy tube $M = (N)_{\eta _p}$. Let $\xi
 '$ be another parallel normal field to $M$. We have seen, perhaps in a smaller $M$, that
$$H^\xi (x) = H^ {\xi '} (x)\, ,$$
for all $x \in M$.

Let $\Cal H$ be the distribution in $M$ perpendicular to the vertical distribution (with respect to $\pi: M\to
N$), i.e. the distribution perpendicular to the leaves $$p \mapsto S(p) =  p + \Phi ^\perp_{\pi (p)}. (p - \pi
(p)) \, .$$

We are around a generic point such  that $(\ker A^M_{\xi} + \, \ker A^M_{\xi '})$ is a distribution
of $M$.

\begin{proposition}\label{isoparam}
Assume that  $\Cal H \subset(\ker  A^M_{\xi} +  \ker A^M_{\xi '}) $. Then, for
all $x\in M$,  $H^\xi (x)= H^{\xi '}(x)$ is an isoparametric submanifold of $\RR^n$.
\end{proposition}

\begin{proof} Observe that $H^\xi (x)= H^{\xi '}(x)$ has flat normal bundle, for all $x \in M$. Indeed,
$$ \nu (H^ \xi (x)) = (\tilde \nu ^ \xi)_{\, | \, H^ \xi (x)} \oplus (\nu M)_{\, |\, H^ \xi
(x)}$$  and both subbundles are parallel and flat (see part (iii) of Proposition~\ref{slice} in \S~\ref{1.6}). From the assumptions, one
obtains that any curvature normal $\bar \eta$ of $H^\xi (x)= H^{\xi '}(x)$ is obtained by one of the following
(non-exclusive) possibilities:
\begin{enumerate}
\item[(a)] $\bar \eta$ is the extension of a curvature normal of a  (isoparametric) fiber $S(x)$ of $\pi : M \to N$.
\item[(b)] $\bar \eta$ is the extension of a curvature normal of a  (isoparametric) fiber of the focalization at
infinity $\pi ^\xi : M \to M_{\tilde \xi}$.
\item[(c)]  $\bar \eta$ is the extension of a curvature normal of a  (isoparametric) fiber of the focalization at
infinity $\pi ^{\xi '} : M \to M_{\tilde {\xi '}}$.
\end{enumerate}
This shows that $\bar \eta$ is parallel in the normal connection and hence $H^\xi (x)= H^{\xi '}(x)$ is an
isoparametric submanifold of $\RR^n$ (see the remark at page ~\pageref{tubolonomico} in \S~\ref{1.5}).
\end{proof}

\medskip

\begin{corollary} There is a compact group of isometries of $\RR^n$, which acts as the isotropy
representation of a simple symmetric space such that (locally) $K.\pi (x)
 = \pi (H^\xi(x))$, for all $x\in M$.  \end{corollary}

\begin{proof} From the above proposition one obtains that $\pi (H^\xi (x))$ is a submanifold with constant
principal curvatures. If $\pi (H^\xi (x))$ is reducible or non-full then $N$ would be reducible or non full
since
$$ N = \underset {y\in \pi (\Sigma _\xi (x))} \bigcup (\pi (H^\xi (x)))_{y- \pi (x)} \text { \  \ \ \
\ \ \ (locally)}$$ and $\pi (\Sigma _\xi (x))\underset {\text {locally}}  { \sim }(\pi (\tilde \nu
^\xi))_{\pi (x)} = \nu _0 (\pi (H^\xi (x)))$ (see Lemma~\ref{hol-tubo}). But the normal holonomy  of $\pi (H ^
\xi (x))$ is irreducible and non-transitive (in the orthogonal complement of the fixed point set). Then, by
making use of Thorbergsson's Theorem \cite {Th}, $\pi (H^\xi (x))$ is a focal manifold of a homogeneous
isoparametric submanifold (we have used that an isoparametric submanifold is always contained in a complete
one \cite {PT}). Then there exists a compact group of isometries $K$ of Euclidean space, acting as the isotropy
representation of a simple symmetric space, and such that (locally) $K.\pi (x)
 = \pi (H^\xi(x))$. This for a fixed $x$. But, for  $x'\neq x$,  $\pi (H^\xi (x'))$ is a parallel manifold, in
 the ambient space, to $\pi (H^\xi (x))$. Since the group $K$ gives the parallel transport in $\nu _0 (K.x)$
 (see Proposition 3.2.4 in \cite {BCO}), one has that $K.\pi (x)
 = \pi (H^\xi(x))$, for all $x\in M$.
\end{proof}

 \medskip

 We summarize the main result in this section, which will be the key tool for the whole article, in the
 following

\begin{theorem}[ tool]\label{main-lemma} Let $N \subset \RR^n$ be a submanifold and assume that its normal holonomy group acts
 irreducibly and non-transitively on the normal space. Let $\eta _q \in \nu _qN$ be a principal vector for the
 normal holonomy action of $\Phi^\perp_q$ on $\nu _qN$.
Let us consider the normal holonomy tube $M: = N_{\eta _q}$, which has flat normal bundle ($\eta _q$ short, in a
neighbourhood of  a generic $q$). \newline Assume that there  exist two non-trivial parallel normal fields $\xi, \ \xi '$
to $N_{\eta _q}$ such that
$$ \Cal H \subset  \ker (A^M_{\xi}) + \ker (A^M_{\xi '})\, ,  $$
where $\Cal H$ is the horizontal distribution in the holonomy tube $M$ (we are assuming, since we are working locally
that the right hand side of the above inclusion, as well as both of its terms, is a $C^\infty$-distribution).
\newline
Then there is a compact group $K$ of isometries of $\RR^n$, acting as the isotropy representation
of an irreducible symmetric space such that, locally around $q$,
$$ N =  \underset {v\in (\nu _0 (K.q))_q}  {\bigcup} (K.q)_v$$
i.e., $N$ is locally, the union of the parallel orbits to $K.q$.

Moreover, $(\nu _0
(K.x))_x$ is contained in the nullity space $\mathcal{N}^{N}_x$ of the second fundamental form $\alpha ^N$ at $x$.
\end{theorem}

\begin{remark}\label{nullity}
We are in  the assumptions of the above theorem.
\newline
(i) The orbit $K.q$ cannot be isoparametric, otherwise $\RR^n = \underset {v\in \nu _q (K.q)}
{\bigcup} (K.q)_v = \underset {v\in (\nu _0 (K.q))_q} {\bigcup} (K.q)_v = N$.\newline
(ii) Observe that  $\text {dim} (\nu _0 (K.x))\geq 1$ (i.e. dimension of the standard fiber) since the
position vector field, from the fixed point of $K$, gives a parallel normal field. So, the nullity is
non-trivial. Note that $\underset {v\in \nu _0 (K.p)} {\bigcup } (K.q)_v $ is globally never a submanifold
since there are always focal parallel orbits. We will come back to this discussion, on the completeness of $N$,
in the case that $N$ is a complex submanifold of $\CC^n$.
\end{remark}

\section{Complex submanifolds of $\CC^n$ with non-transitive normal holonomy}\label{3}

Let $N \subset \CC^n$ be a complex (not necessarily complete) submanifold which is irreducible and full.
The standard complex structure of $\CC^n$ is denoted, as usual, by $J$.\\

Then, by \cite {DS}, the normal holonomy group acts irreducibly on the normal space. Assume furthermore, that \textit{the normal holonomy group  of $N$ is non-transitive} on the normal sphere. Let $\Phi ^\perp_q$ be the normal holonomy group at $q \in N$, which acts by complex transformations on $\nu _q N$. Choose $ \xi ^1_q \in \nu _q N$ such that the orbit
 $\Phi ^\perp_q.\xi^1 _q$ projects down to the (unique) complex orbit in the (complex) projectivization of the normal space
 $P(\nu _qN)$ of $\nu _q N$ (see \cite{BW}). This implies that the orthogonal complement of $\xi ^1_q$ in the normal space
 of  the holonomy orbit,
  $(\xi ^1 _q) ^\perp
 \cap  \nu _ {\xi ^1 _q} \Phi ^\perp_q.\xi^1 _q $ is a complex subspace of  $\nu _q N$. Since the normal
 holonomy is not transitive on the sphere, $(\xi ^1 _q) ^\perp
 \cap  \nu _ {\xi ^1 _q} \Phi ^\perp_q.\xi^1 _q $ is not a trivial subspace.

  Now choose
 $\xi ^2 _q \neq 0$ which lies in $(\xi ^1_q)^\perp \cap  \nu _ {\xi ^1 _q} \Phi ^\perp_q.\xi^1 _q $.

 Since $R^\perp_{X,Y}$ always lies in the holonomy algebra one gets that $0 = \langle R^\perp _{X,Y}\xi ^1 _q ,
 \xi^2
_q\rangle$. So, by the Ricci identity, $[A^N_{\xi ^1 _q}, A^N_{\xi ^2_q}] = 0$. The same is true if we replace
$\xi ^2_q$ by $J\xi ^2_q$. So, $A^N_{\xi ^1 _q}$ also commutes with $A^N_{J\xi ^2_q}$.

By the well-known formulae of complex geometry $A^N_{J\xi ^2_q}= -JA^N_{\xi ^2 _q}$ and $J$ anti-commutes with all shape operators. So,
$$[A^N_{\xi ^1 _q}, A^N_{J\xi ^2_q}] = J (A^N_{\xi ^1 _q}A^N_{\xi ^2 _q} + A^N_{\xi ^2 _q}A^N_{\xi ^1 _q}) =0.$$
But
$$[A^N_{\xi ^1 _q}, A^N_{\xi ^2_q}] =  A^N_{\xi ^1 _q}A^N_{\xi ^2 _q} - A^N_{\xi ^2 _q}A^N_{\xi ^1 _q} =0\, .$$
Then \begin{equation}\label{(A)}A^N_{\xi ^1 _q}A^N_{\xi ^2 _q }= A^N_{\xi ^2 _q}A^N_{\xi ^1 _q} = 0\, .  \end{equation}
We may assume that  the slice representation orbit $(\Phi ^\perp_q)_{\xi ^1 _q}.\xi ^2 _q$ is a principal one
in the normal space to the holonomy orbit,
 where $(\Phi ^\perp_q)_{\xi ^1 _q}$ is the isotropy subgroup at $\xi ^1_q$. Observe that we can find such a
 $\xi ^2 _q$, since $\xi ^1_q$ is a fixed point for the slice representation of $(\Phi ^\perp_q)_{\xi ^1 _q}$.

Observe, by construction, that one has also that $$A^N_{\tau ^\perp _c (\xi ^1 _q)}A^N_{\tau ^\perp _c(\xi ^2
_q) }= 0$$ where $\tau ^\perp _c$ is the normal parallel transport along any arbitrary curve $c$ in $N$ which
starts at $q$.

Consider the iterated holonomy tube  $$(N_{\xi ^1 _q})_{\xi ^2_q},$$ which coincides with the full holonomy tube
$N_{\zeta _q}$, where $\zeta _q= \xi ^1 _q + \xi ^2 _q$ (see the theorem in Appendix of \cite{O2}).
Of course we have to choose $\xi ^1_q$ short and after that  $ \xi ^2 _q$ short enough. The vector
$\xi ^1_q$ gives rise to a parallel normal field $\tilde \xi$ to the partial holonomy tube $N_{\xi ^1 _q}$ so
that $(N_{\xi ^1 _q})_{\tilde \xi} = N$ (see \S~\ref{1.5}). This parallel normal field can be lifted to a
parallel normal field $\xi$ of $(N_{\xi ^1_q})_{\xi ^2_q} = N_{\zeta _q}$. We can do so, since $\tilde \xi (x)$
is fixed by the normal holonomy group of $N_{\xi ^1_q}$ at $x$ and hence it is perpendicular to  any holonomy
orbit.
  Similarly, $\xi ^2_q$ gives rise to a parallel normal field $\xi '$ in $(N_{\xi ^1_q})_{\xi ^2_q} = N_{\zeta _q}$.

  By (\ref{(A)}) and the tube formulae relating shape operators of parallel focal manifolds (see \S~\ref{sollev}) one obtains
  that
 $$ A^M _{\xi}A^M_{\xi '\, | \, \Cal H} =0$$
 where $M:=  N_{\zeta _q}$ and $\Cal H$ is the horizontal distribution on $M$. Clearly, by (\ref{(A)}) we also have $A^M _{\xi'}A^M_{\xi \, | \, \Cal H} =0$. Therefore $A^M_{\xi \, | \, \Cal H}$ and $A^M _{\xi'  \, | \, \Cal H}$ are simultaneously diagonalizable, so $\Cal H \subset(\ker  A^M_{\xi} +  \ker A^M_{\xi '}) $.

 \medskip

 By Theorem~\ref{main-lemma}  in the previous section, one has the following

\begin{theorem}\label{main-cx} Let $N\subset \CC^n$ be a complex irreducible and full
 submanifold such that the normal holonomy group (which must act irreducibly by \cite{DS}) is not transitive on the unit sphere of the normal space.  Then there is a group $K$, acting as the isotropy representation of an irreducible Hermitian symmetric
 space, such that $N$ is locally given, around a generic point $q$, as
$$ N =  \underset {v\in (\nu _0 (K.q))_q}  {\bigcup} (K.q)_v\, .$$
Moreover $(\nu _0 (K.p))_p$ is contained in the nullity space $\mathcal{N}^{N}_p$ of the second fundamental form $\alpha ^N$ at $p$.
.
\end{theorem}

\begin{proof}  It remains only to show that $K$ is of Hermitian type. We may assume that the origin $0\in
\C^n$ is the fixed point of $K$. If $p\in N$, then the position vector $ \overrightarrow{p}$, by the description
given before, belongs to $T_pN$. So,  $i\overrightarrow{p} \in T_pN$. Then the orbits of the $S^1$ action $(t,
x) \mapsto e^{it}x$ on $\CC^n$ are tangent  to $N$ at the points of $N$. This implies that $N$ is (locally)
$S^1$-invariant. Let now $\bar K$ be the subgroup of linear isometries of $\CC^n$ generated by $K$ and $S^1$.
Then $\bar K.p  \subset N$ and so $\bar K$ is not transitive in the sphere. By the Theorem of Simons \cite {S, O1}, since $K$ acts irreducibly, one must have $\bar K = K$ and so $K$ is of Hermitian type.
\end{proof}

As a corollary, we are now ready to prove the Berger type theorem for submanifolds of $\CC^n$



\begin{proof}[Proof of  Theorem~\ref{berger-c}] We are in the assumptions at the beginning of this section. If the normal holonomy group of $N$ is not transitive, then, locally,
$$ N =  \underset {v\in (\nu _0 (K.q))_q}  {\bigcup} (K.q)_v$$
where $K$ acts as in the previous theorem, and in particular it is irreducible (we assume that $0$ is the fixed
point of $K$). Recall we are assuming that $N$ is complete (not necessarily immersed), so, if $p\in N$, since $N$ is analytic, then
the line $t \mapsto tp$ is contained in $N$ (i.e. this line is the image, via the immersion, of a geodesic in
$N$. In order to simplify the notation we avoid the immersion map). From the construction, for all $t$, $T_{tp}N
= T_pN$, as subspaces of $\CC^n$. So, the isotropy $K_{tp}$ must leave this subspace invariant. A
contradiction for $t=0$, since $K$ acts irreducibly. Thus the normal holonomy group must be transitive.
\end{proof}


\section{Complex submanifolds of $\C P^n$\\with non-transitive normal holonomy}\label{4}

Let $M \subset \C P^n$ be a full complex submanifold of the complex projective space. Let $\mathcal{N}^{M}_p = \{ X_p \in T_pM : \alpha^M(\cdot, X) = 0 \}$ be the nullity of the second fundamental form $\alpha^M$ at $p \in M$.

The goal of this section is to prove the following theorem.

\begin{theorem} \label{main} Let $M \subset \C P^n$ be a complex, complete and full submanifold. If the normal holonomy group $\Phi_p^{\perp}$ does not act transitively on the unit sphere of the normal space $\nu_p(M)$ at $p \in M$ then $M$ is the complex orbit, in the complex projective space, of the isotropy
representation of an irreducible Hermitian symmetric space.
\end{theorem}

We start with the following

\begin{lemma} \label{lem2} Let $M \subset \C P^n$ be a complex submanifold and let $\widetilde{M} \subset \C^{n+1}$ be its lift to $\C^{n+1}$. Assume that the tangent vector  $\widetilde{v}_{\widetilde{p}} \in T_{\widetilde{p}} \widetilde{M}$ is not a complex multiple of the position vector $\widetilde{p}$. If $\widetilde{v}_{\widetilde{p}} \in \mathcal{N}^{\widetilde{M}}$ then its projection $v_p$ to $T_pM$ belongs to the nullity of the second fundamental form of $M$, i.e. $v_p \in \mathcal{N}^M$.
\end{lemma}

\begin{proof}  We can assume that $\widetilde{v}_{\widetilde{p}} \in T_{\widetilde{p}} \widetilde{M}$ is horizontal with respect to the submersion $\pi: \widetilde{M} \rightarrow M$. Let $\widetilde{v} \in \Gamma(T\widetilde{M})$ be a horizontal and projectable vector field that extends $v_p$. Let $\widetilde{X} \in \Gamma(\widetilde{M})$ be an arbitrary horizontal and projectable vector field defined around $\widetilde{p} \in \widetilde{M}$.
From equation (\ref{fun}) we get
\[(D_{\widetilde{X}}\widetilde{v})_{\widetilde{p}} = \widetilde{{\nabla^{FS}_Xv}} + {\mathcal{O}}(\widetilde{X},\widetilde{v})  .\]
So
\[ (\nabla^{\widetilde{M}}_{\widetilde{X}} \widetilde{v})_{\widetilde{p}} + \alpha^{\widetilde{M}}(\widetilde{X}, \widetilde{v}) =  \widetilde{{\nabla^{M}_Xv}} + \widetilde{\alpha^M(X,v)} + {\mathcal{O}}(\widetilde{X},\widetilde{v})  .\]
Taking normal and tangent components with respect to $\widetilde{M}$ we get  $\alpha^{\widetilde{M}}({\widetilde{X}}_{\widetilde{p}},\widetilde{v}_{\widetilde{p}}) = \widetilde{\alpha^M(v_p,X_p)}$. Thus, if $\widetilde{v}_{\widetilde{p}} \in \mathcal{N}^{\widetilde{M}}$, then $v_{p} \in \mathcal{N}^{M}$.
\end{proof}

\begin{lemma} \label{lem1} Assume that $M \subset \C P^n$ is full and its normal holonomy group does not act transitively on the normal space $\nu_p(M)$. Then the normal holonomy group of $\widetilde{M}$ does not act transitively on $\nu_{\widetilde{p}}(\widetilde{M})$, where $\pi(\widetilde{p}) = p$.
\end{lemma}

\begin{proof}  Let $\widetilde{R}^{\perp}$ be the curvature tensor of the normal connection of $\widetilde{M}$.
Notice that the Ricci equation implies $\widetilde{R}^{\perp}_{X \, , \, \cdot } = 0$ if $X \in
\mathcal{N}^{\widetilde{M}}$. So we can use the Lemma in the appendix of \cite{O2}. Namely, any normal parallel
transport $\tau^{\perp}_{\widetilde{\gamma}}$ along a loop $\widetilde{\gamma}(t)$ starting at $\widetilde{p}$
can be written as $\tau^{\perp}_{\widetilde{\gamma}} = \tau ^{\perp}_{v} \circ \tau ^{\perp}_{\widetilde{c}}$,
where $v$ is a vertical curve (i.e., $\frac{dv}{dt} \in \mathcal{N}^{\widetilde{M}}$) and $\widetilde{c}$ is a
horizontal, that is to say, $\frac{d \widetilde{c}}{dt} \in (\mathcal{N}^{\widetilde{M}})^{\perp}$. Notice that $\tau_{v}$
is just an Euclidean translation, i.e., $\tau_v( \xi_q ) = \xi_p$, where $\xi_p = \xi_q$ as vectors of
$\C^{n+1}$. Observe also that the horizontal curve $\widetilde{c}$ is the lift (starting at $\widetilde{p} \in
\widetilde{M}$) of a loop $c$ starting at $p \in M$. Let $\widetilde{\xi}_{\widetilde{p}} \in
\nu_{\widetilde{p}}(\widetilde{M})$ be a normal vector and let $\xi_p \in \nu_p(M)$ its projection to $M$. Let
$\xi(t)$ be the normal parallel transport along the loop $c$. Then by using equation (\ref{fun}) it is not
difficult to see that the horizontal lift $\widetilde{\xi}(t)$ is parallel with respect to the normal connection
along the curve $\widetilde{c}$.

Observe that $\widetilde{c}(1)$ belongs to the sphere of radius $||\widetilde{c}(0)||$, and $\pi (
\widetilde{c}(0)) = \pi (\widetilde{c}(1)) = p$. So, $\widetilde{c} (1) = e^{i\theta}$, for some $ \theta  \in
[0, 2\pi)$. Since the isometry $ x \mapsto e^{i \theta} x$ of $\widetilde{M}$ projects down to the identity of
$M$, one has that $e^{i\theta} \tau ^{\perp}_{\widetilde{c}} = e^{i \theta} \tau ^{\perp}_{\widetilde{\gamma}}$,
which coincides, via $d\pi_{|\widetilde p}$, with $\tau ^{\perp}_{c}$.  Since any  $e^{i \theta}$ belongs to the
normal holonomy group of $\widetilde{M}$ (recall that the normal holonomy acts as an $s$-representation; see
\cite[Remark 2.2]{DS}) we can conclude that  the normal holonomy groups of $M$ and $\widetilde{M}$ identify (via $d\pi_{|\widetilde
p}$). Thus, if one of them does not act transitively on the unit sphere neither does the other.
\end{proof}

\begin{remark}\label{orb} Notice that two orbit equivalent Hermitian $s$-representations are equivalent. Since the normal holonomy groups of $M$ and $\widetilde{M}$ act as Hermitian $s$-represen\-tations, the above proof shows that the holonomy representations are indeed equivalent. Roughly speaking, the holonomy groups of $M$ and $\widetilde{M}$ are equal.
\end{remark}





Now we are ready to prove Theorem~\ref{main} and therefore  Theorem~\ref{berger-cp}.  The main tool is Theorem~\ref{main-cx}.

\begin{proof} [Proof of  Theorem~\ref{main}]  Notice that Lemma \ref{lem1} allows us to apply Theorem~\ref{main-cx} to
$N = \widetilde{M}$. So we get that \[ \widetilde{M} = \bigcup_{{v\in (\nu _0 (K.q))_q}} (K.q)_v \, \, \, , \]
where $K$ is the isotropy group of a irreducible Hermitian symmetric space. Observe also that $\nu _0 (K.q)_q$  is contained in the nullity of the second fundamental form of the complex submanifold $N$ at $q$ (which is a complex subspace). Then Lemma~\ref{lem2} and Theorem~\ref{no-abe} in the Appendix imply that $\dim (\nu _0 (K.q)_q) = 1$, otherwise the nullity of the second fundamental form of $M$ would be not trivial. In fact any element in $\nu _0 (K.q)_q$ that is perpendicular to the position vector $q$ cannot be a complex multiple of $q$ (since $i \, q$ is tangent to $\widetilde M$). Since $\widetilde{M}$ is full we get that the unique fixed point of $K$ is the origin $0 \in \C^{n+1}$. So the leaves of the nullity distribution $\mathcal{N}^{\widetilde{M}}$ are just the complex lines given by the fibers of the submersion $\pi: \widetilde{M} \rightarrow M$.
Thus,  $K$ acts transitively on the complex submanifold $M \subset \C P^n$. Therefore, $M$ is a complex orbit of the projectivization of an irreducible Hermitian $s$-representation (cf. \cite{CD}).
\end{proof}

\section{Further comments}\label{further}

We now explain why the completeness assumption cannot be dropped either  in  Theorem \ref{berger-cp} or in Theorem \ref{berger-c} .\\
Let $M$ be a submanifold of Euclidean space and let $$ N =  \underset {v\in \nu _0 (M)_q}
{\bigcup} M_v$$ (defined locally around $M$, $v$ short enough). It is standard to show that the normal holonomy of $N$
at $q$ coincides with the semisimple part of the normal holonomy of $M$ at $q$. Let now $K$ act on
$\C^{n+1}$ as an irreducible Hermitian $s$-representation of rank $r$ and  let $v \in \C^{n+1}$ be such that
$K.v$  projects down to a complex orbit of $\C P^n$. Choose a short enough normal vector $\xi \neq 0$ to $K.v$ at
$v$ that it is perpendicular to $v$. Moreover, assume that the normal holonomy orbit of $\Phi ^ \perp  . \xi =
K_v.\xi$ is a complex submanifold of the projectivization of the  semisimple normal space $\nu _0 (K.v)^ \perp
= v^\perp \subset \nu (K.v) $. By [HO], the dimension over $M = K.(v+\xi)$ of $\nu_0 (K.(v+\xi))$ is $2$.
Moreover,  the semisimple part of the normal holonomy representation of  $K .  (v+\xi))$ at $v+\xi$ has rank $r
-2$. From the above choice of $v$ and $\xi$,  it is not hard to see that this semisimple part of the normal
holonomy representation is of Hermitian type.  Defining $N$ like at the beginning of this discussion one has that $N$ is
a complex submanifold of  $\C^{n+1}$ with not transitive irreducible normal holonomy, if $r \geq 4$. Moreover, $N$ projects down to the projective space $\C P^n$ as a complex submanifold $\bar N$ with non transitive
holonomy (see Lemma \ref{lem1} and its proof). Notice however that $\bar N$ cannot be extended to a complete complex
submanifold. Indeed,  the  second fundamental form has nullity on an open set  and so $\bar N$ cannot be
homogeneous as it would follow from  Theorem \ref{berger-c} (see Theorem \ref{no-abe} in the Appendix). This shows that the assumption of completeness cannot be dropped.

\medskip

We would like also remark that our main results are far from being true for (non necessarily complex) submanifolds of Euclidean space $\RR^n$. For example, a submanifold with flat normal bundle is not necessarily homogeneous. Anyway, it is an open problem if compact homogeneous submanifolds  whose normal holonomy is not transitive on the unit sphere of the normal space, are orbits of $s$-representations (cf. \cite[Conjecture 6.2.14, page 198]{CDO}).

\section*{Appendix}\label{appendix}

The aim of this appendix is to prove the following

\begin{theorem}\label{no-abe}
Let $M$ be a complete full complex submanifold of $\C P^n$ whose normal holonomy is not transitive, then the second fundamental form has no nullity in an open subset of $M$.
\end{theorem}

Before giving the proof, we note that the above result is also a consequence of the fact  that complete complex submanifolds of the projective space have no nullity at some open subset (see \cite[Theorem 3]{AbMa}). We include a proof for the sake of being self contained.

\begin{proof}
From Lemma~\ref{lem1} and Remark~\ref{orb}, we have that the normal holonomy groups of $M$ and $\tilde M$ are equal. Then we can apply Theorem~\ref{main-cx} to $\tilde M$. We will show that
there are other singular points different from $0$, if $K.q$ is not most singular (in the hypothesis and
description of the above theorem, and $0$ is the fixed point of $K$). Assume $M$ to be complete and  let $q\in M$
and let $\eta \in (\nu _0 K.q)_{q}$, not a multiple of the position vector $\overrightarrow {q}$. and identify
$\eta$ with a parallel normal field to $K.q$.  The shape $A^{K.q}_\eta$ has constant eigenvalues $\lambda _1,
\cdots , \lambda _g$. Let $E_1, \cdots , E_g$ be the associated eigendistributions on $K.q$. We may assume that
all the eigenvalues are different from $0$, by adding to $\eta$ a small multiple of the position vector. The
isotropy subgroup $K_i : = K_{q + \lambda _i^{-1}\eta}$, which is bigger that $K_q$, must act transitively on
the integral manifold $S_i(q) \subset K.q$ through $q$ of the eigendistribution $E_i$, $i=1, \cdots , g$.  So,
the subgroup $\bar K$ of $K$ generated by the isotropy subgroups $K_1 , \cdots , K_g$ acts transitively on
$K.q$. But, from the description of Theorem~\ref{main-cx}, one has that $T_{q + t\eta}M = T_qM$ and so $K_iT_qM =
T_qM$, since isotropy subgroups preserve tangent spaces of invariant submanifolds. But the tangent space of $M$
do not change if one moves along $(\nu _0M)_q$. Then
$$\underset {x\in M} {\bigcup} T_xM = \bar K T_qM = T_qM$$

A contradiction.

Then $M$ is not smooth at some $q + \lambda _i^{-1}\eta \neq 0$. This singularity projects down to the
projective space. Thus $M$ would be not complete.
\end{proof}

 \bibliographystyle{amsplain}

\end{document}